\documentclass[12pt,a4paper]{article}

\usepackage{url}
\usepackage{hyperref}
\usepackage{amssymb,amsmath,amsthm,mathrsfs,amsfonts,bbm}

\theoremstyle{plain}
\newtheorem{lem}{Lemma}[section]

\newtheorem{thmx}{Theorem}

\newtheorem{prop}[lem]{Proposition}
\newtheorem{exa}[lem]{Example}

\theoremstyle{definition}
\newtheorem{defn}{Definition}[section]

\theoremstyle{remark}
\newtheorem{rem}{Remark}[section]

\numberwithin{equation}{section}

\newcommand{\N}{\mathbb{N}}
\newcommand{\R}{\mathbb{R}}

\newcommand{\La}{\Lambda}

\newcommand{\calL}{\mathcal{L}}

\newcommand{\calG}{\mathcal{G}}
\newcommand{\calI}{\mathcal{I}}

\newcommand{\w}{\omega}

\begin{document}

\title{Weak Mixing in switched systems\thanks{Accepted for publication in SCIENCE CHINA Mathematics}
}
\author{Yu Huang, Xingfu Zhong\\
\\
Department of Mathematics, Sun Yat-Sen University\\
Guangzhou, 510275 P. R. China\\
e-mail: stshyu@mail.sysu.edu.cn(Yu Huang)\\
zhongxingfu224634@163.com(Xingfu Zhong)}

\maketitle
\date
\begin{abstract}
Given a switched system, we introduce weakly mixing sets of type $1,2$ and Xiong chaotic sets of type $1,2$ with respect to a given set and show that they are equivalent respectively.
\end{abstract}

\section{Introduction}

A switched system consists of a family of subsystems and a rule that governs the switching among them.
More precisely, letting $X$ be a metric space not necessarily compact and ${\cal G}=\{f_0,f_1,\ldots\}$ a
family of countable continuous self-maps of $X$, we consider the discrete-time dynamical system in the form of
\begin{equation}\label{eq1.1}
   x_{n+1}=f_{\omega_n}(x_n),
\end{equation}
where $x_n\in X$, $\omega_n$ takes a value in the finite-symbolic
set ${\cal I}\triangleq \{0,1,\cdots\}$. If we denote the set (also called symbolic space) of all mappings $\mathbb{N}\rightarrow {\cal I}$ by
\[
  {\cal I}^{\mathbb{N}}=\{\omega:\ \mathbb{N}\rightarrow {\cal I} \},
\]
then switching can be classified into two situations: (i) arbitrary switching; that is, the
switching rule can be taken arbitrarily from ${\cal I}^{\mathbb{N}}$; (ii) switching is subject to
certain constraints; i.e., the switching rule is characterized by a subset of ${\cal I}^{\mathbb{N}}$.

Switched systems are found in many practical systems, see \cite{Costa-Fragoso-Marques} and \cite{Sun-Ge2005}. When the switchings are arbitrary,
one can take the switched system (\ref{eq1.1}) as a free semigroup action $G$ generated by ${\cal G}$, i.e., $G=\bigcup_{n\in \mathbb{N}}G^n$,
$G^n=\{f_{\omega_{n}}\cdots f_{\omega_{1}} \mid \omega_i\in {\cal I},\ i=1,\cdots, n\}$.
There are many works for dynamical systems under semigroup actions: topological entropy~\cite{Bis2004,Bufetov1999,Wang-Ma2015,Wang-Ma-Lin2016}; transitivity, mixing, and chaos~\cite{cairns2007topological,wang2014almost,wang2017weakly,zeng2017multi,wang2015m,Ma2017Some}; sensitivity~\cite{kontorovich2008note,
polo2010sensitive}; shadowing property~\cite{Bahabadi2015Shadowing,Ma2017Some}; specification property~\cite{Rodrigues2016Specification,Ma2017Some}; we refer the reader to the references therein for other related investigations.

Let $(X,T)$ be a topological dynamical system(TDS), where $X$ is a compact metric space and $T$ is a continuous self-map of $X$. Given two open sets $U,V$ of X, the \emph{hitting time set} of $U,V$ is defined as $N(U,V)=\{n\in\N:U\cap T^{-n}(V)\neq\emptyset\}$. The notion of weak mixing for a TDS is classical.
Recall that $(X,T)$ is \emph{weakly mixing} if for any four non-empty open subsets $U_1,U_2,V_1,V_2$ the set $N(U_1,V_1)\cap N(U_2,V_2)$ is not empty. Xiong and Yang in~\cite{Xiong1991Chaos} gave a characterization of weak mixing by chaotic sets. (After more than twenty years, Oprocha~\cite{Oprocha2013Coherent} showed a shorter proof to help us pass through it.) Recall that a subset $K$ of $X$ is said to be a \emph{Xiong chaotic set} with respect to a given increasing sequence $\{p_i\}$ if for any subset $F$ of $K$ and any continuous map $\phi:F\to X$ there exists an increasing subsequence $\{q_i\}$ of $\{p_i\}$ such that $\lim_{i\to\infty}T^{q_i}(x)=\phi(x)$ for every $x\in F$~\cite{Li2014Localization}. A subset $K\subset X$ is called \emph{Xiong chaotic} if it is Xiong chaotic with respect to $\{n\}$. Blanchard and Huang~\cite{Blanchard2017Entropy} introduced the notion of weakly mixing set and proved a TDS with positive entropy has many weakly mixing sets. Recall that a nondegenerate closed subset $A$ of $X$ is weakly mixing if for any $n\in\N$, any non-empty open subsets $U_1,U_2,\cdots,U_n$ and $V_1,V_2,\cdots,V_n$ of $X$ satisfying $U_i\cap A\neq\emptyset$ and $V_i\cap A\neq\emptyset$ for $i=1,2,\cdots,n$, there exists $m\in\N$ such that $A\cap U_i\cap T^{-m}(V_i)$ is non-empty set for $i=1,2,\cdots,n$. Recently, mixing sets of finite order and relative dynamical properties are widely studied, see~\cite{Li2014Localization,Oprocha2011On,oprocha2012sets,
oprocha2013weak,li2015recurrence,balibrea2012weak}. There are many other notions that are stronger than weak mixing, such as strong mixing, $\Delta$-transitivity, and $\Delta$-mixing. It is known that $\Delta$-transitivity implies weak mixing~\cite{moothathu2010diagonal}. We refer the reader to~\cite{chen2014multi,chen2015point,zeng2017multi} for more about $\Delta$-transitivity. Later, Huang et al. introduced the concept of $\Delta$-weakly mixing set~\cite{huang2017positive} and proved that if a TDS has positive topological entropy, then
there are many $\Delta$-weakly mixing sets.

On the other hand,
We find that the definitions of weakly mixing in~\cite{Ma2017Some} and~\cite{cairns2007topological,wang2017weakly,zeng2017multi} are slightly different, which are called weak mixing of type 1,2 (for short WM1 and WM2) in this paper.
\begin{itemize}
  \item[\textbf{WM1}] Under the setting of~\cite{Ma2017Some}, let $G_1=\{f_0,\ldots,f_{m-1}\}$. Let $G$ be the semigroup generated by $G_1$. Then semigroup action $G$ is said to be \emph{weakly mixing of type 1} if for any nonempty open subsets $U_1,U_2,V_1,V_2$ of $X$, there exists $n\in\N$ with $n>0$, $f_{n-1}\circ\cdots\circ f_0$ and $g_{n-1}\circ\cdots\circ g_0$ with $f_i,g_i\in G_1$, $i\in\{0,\ldots,n-1\}$ such that $f_{n-1}\circ\cdots\circ f_0(U_1)\cap V_1\neq\emptyset$ and $g_{n-1}\circ\cdots\circ g_0(U_2)\cap V_2\neq\emptyset$.
  \item[\textbf{WM2}] Under the view of~\cite{cairns2007topological,wang2017weakly,zeng2017multi}, we say that $(X,G)$ is \emph{weakly mixing of type 2} if for any nonempty open subsets $U_1,U_2,V_1,V_2$, there exists $s\in G$ such that $s(U_1)\cap V_1\neq\emptyset$ and $s(U_2)\cap V_2\neq\emptyset$.
\end{itemize}
It is clear the $WM2$ implies $WM1$.  Hui and Ma~\cite{Ma2017Some} gave some characterizations for WM1 under a constraint: semigroup has shadowing property, and
Zeng~\cite{zeng2017multi} gave some characterizations for WM2 under a condition: every element of the semigroup $G$ is surjective.

Inspired by the work of Blanchard and Huang~\cite{Blanchard2017Entropy}, we respectively introduce weak mixing sets and Xiong chaotic sets of type 1,2 with respect to a given set for switched systems.
Let $\La\subset\calI^\N$.
If the system (\ref{eq1.1}) is subjected to $\La$, then we denote the switched system by $(X,\calG|_\La)$. Our aim is to show that weakly mixing sets and Xiong chaotic sets of type 1,2 with respect to a given set for $(X,\calG|_\La)$ under the two different views respectively are equivalent (Theorem~\ref{thm:cha-of-mixing}).

\section{WM1 and WM2 for a switched system}\label{Sec:definition}

Let ${\cal I}=\{0,1,\cdots\}$.
For $n\in\mathbb{N}$ we denote by ${\cal I}^n$ the set of words of $\calI$ of length $n$, i.e.,
${\cal I}^n=\{u=(\omega_{0}\cdots\omega_{n-1})\mid \omega_i\in {\cal I},\ \ i=1,\cdots n-1\}$. Let ${\cal I}^*=\cup_{n\geq 1}{\cal I}^n$ be the set of all words of ${\cal I}$.

Suppose $\La$ is a subset of $\calI^\mathbb{N}$. The pre-language of $\La$ is defined by ${\cal L}(\La)=\{u\in {\calI}^*:\exists x\in\La, i\in\N~s.t.~ u=x_{0}\cdots x_{l(u)-1}\}$, where $l(u)$ denotes the length of $u$. We denote by
${\cal L}^n(\La)={\cal L}(\La)\cap {\cal I}^n$ the pre-language of words of length $n$.

Let $X$ be a metric space with metric $d$, which is not necessarily compact. Consider a family of countable continuous self-maps ${\cal G}=\{f_i\}_{i\geq0}$ of $X$.
Given a subset $\La$ of $\calI^\N$, we consider the switched system (\ref{eq1.1}) under the switching sequences subjected to $\La$. This system can be characterized
by the action $G=\cup_n G_n$ on $X$ where $G_n=\{f_{\omega}=f_{\omega_{n-1}}\circ\cdots\circ f_{\omega_0}\mid \omega=(\omega_{n-1}\cdots\omega_0)\in {\cal L}^n(\La)\}$.
We remark that the action $G$ is not a free semigroup acting on $X$ when $\La$ is a proper subset of ${\cal I}^\mathbb{N}$.

Now, we introduce the notions of WM1 and WM2 with respect to a given set for switched systems and give the main theorem in this paper.

For every two non-empty open subsets $U,V$ of $X$, the \emph{hitting time set of type $1$} and \emph{hitting time set of type 2} of $U$ and $V$ are respectively defined as
\[N_1(U,V)=\{n\in\N:\exists\w\in\calL^n(\La)~\text{s.t.}~f_\w(U)\cap V\neq\emptyset\},\]
\[N_2(U,V)=\{\w\in\calL(\La): f_\w(U)\cap V\neq\emptyset\}.\]
If no misunderstanding is possible, we omit the subscripts.
\begin{defn}\label{defn:MD}
Let $(X,\calG|_\La)$ be a switched system. Let $K$ and $Q$ be two subsets of $X$.
For $n\in\N$ and $n\geq2$,
if for any non-empty open subsets $U_1,U_2,\ldots, U_n$, $V_1,V_2,\ldots,V_n$ of $X$
with $U_i\cap K\neq\emptyset$ and $V_i\cap Q\neq\emptyset$ for $i=1,2,\cdots,n$,
there exists $S=\{q_i\}$ in $\N$ such that
\[S\subset N_1(K\cap U_i,V_i),~for~i=1,2,\ldots,n,\]
then we say that $K$ is \emph{weakly mixing of type $1$ with respect to $Q$ of order $n$}; if
there exists an infinite sequence $S=\{s_i\}$ in $\calL(\La)$ with $|l(s_i)|$ increasing such that
\[S\subset N_2(K\cap U_i,V_i),~for~i=1,2,\ldots,n,\]
then we say that $K$ is \emph{weakly mixing of type $2$ with respect to $Q$ of order $n$};
The set $K$ is said to be \emph{weakly mixing of type $1$, $2$ with respect to $Q$} respectively if $K$ respectively is weakly mixing of type $1$, $2$ with respect to $Q$ of order $n$ for every $n\geq2$.
\end{defn}

\begin{rem}
Suppose $X$ is compact.
(\romannumeral1) If $X=K=Q$ and $|\calG|=1$, then the definitions of WM1 and WM2 coincide with the definition of weak mixing in classical dynamical systems.

(\romannumeral2) If $K=Q\subset X$ and $|\calG|=1$, then these definitions and the concept of weak mixing sets in dynamical system(see~\cite{Blanchard2017Entropy}) coincide.

(\romannumeral3) If $X=K=Q$ and $|\La|$, then this two definitions and weakly mixing of all orders in~\cite{balibrea2012weak} coincide.
\end{rem}

\begin{prop}
Let $\La=\calI^\N$. If $f_0,f_1,\ldots$ commute with each other, then the following statements are equivalent:
\begin{itemize}
  \item[(\romannumeral1)] $X$ is weakly mixing of type $2$ with respect to itself;
  \item[(\romannumeral2)] $X$ is weakly mixing of type $2$ with respect to itself of order $2$;
\end{itemize}
\end{prop}

\begin{proof}
It is clear that (\romannumeral1) implies (\romannumeral2). Given nonempty subsets $U_1,U_2,V_1,V_2$, pick $s\in\calL(\calI^\N)$ from the set $N(U_1,U_2)\cap N(V_1,V_2)$.
Put
\[U=U_1\cap f_s^{-1}(U_2),~~V=V_1\cap f_s^{-1}(V_2).\]
Then for any $w\in N(U,V)$, we have $w\in N(U_1,V_1)$ and
\begin{align*}
 \emptyset\neq f_s(f_w(U)\cap V)\subset f_w(f_s(U))\cap f_s(V)\subset f_w(U_2)\cap V_2.
\end{align*}
Therefore, $N(U,V)\subset N(U_1,V_1)\cap N(U_2,V_2)$. It follows that for any $n\geq2$ and $U_1,\ldots,U_n,V_1,\ldots,V_n$, there exist nonempty subsets $U_{n+1},V_{n+1}$ such that
\[N(U_{n+1},V_{n+1})\subset\cap_{i=1}^nN(U_i,V_i).\]
It remains to show that $N(U_{n+1},V_{n+1})$ is infinite. Choose $s\in N(U_{n+1},V_{n+1})$. Then
$f_s^{-1}(V_{n+1})\cap U_{n+1}\neq\emptyset$. Hence $N(U_{n+1},f_s^{-1}(V_{n+1})\cap U_{n+1})\neq\emptyset$. Pick $w\in N(U_{n+1},f_s^{-1}(V_{n+1})\cap U_{n+1})$ and $u\in U_{n+1}$ such that $f_w(u)\in f_s^{-1}(V_{n+1}$. Then $f_{ws}(u)\in V_{n+1}$. It follows that $ws\in N(U_{n+1},V_{n+1})$. Repeating the process completes the proof.
\end{proof}

\begin{defn}
Let $(X,\calG|_\La)$ be a switched system. Suppose $C$ and $Q$ are two subsets of $X$.
We say that $C$ is \emph{Xiong chaotic of type 1} with respect to $Q$ if for any subset $E$ of $C$ and any continuous map $g:E\to Q$ there exists an increasing unbounded sequence $\{q_i\}$ in $\N$ such that for every $x\in E$ there is a sequence $\{w_{q_i}^x\}$ with $w_{q_i}^x\in\calL^{q_i}(\La)$ such that $\lim_{i\to\infty}\phi(q_i,x,\w_{q_i}^x)=g(x)$;  \emph{Xiong chaotic of type 2} with respect to $Q$  if for any subset $E$ of $C$ and any continuous map $g:E\to Q$, there exists an increasing unbounded sequence $\{q_i\}$ of positive integers and $\w_{q_i}\in\calL^{q_i}(\La)$ such that $\lim_{i\to\infty}f_{\w_{q_i}}(x)=g(x)$ for every $x\in E$.
\end{defn}

Following the idea in~\cite{Li1975Period},
we call a pair $(x,y) \in X \times X$ \emph{scrambled of type 1} if there exists two infinite sequences $\{\w_i\}$ and $\{s_i\}$ with $\w_i, s_i\in\calL^i(\La)$ such that
\[\liminf_{i\to\infty}d(f_{\w_i}(x),f_{s_i}(y))=0~ ,~\limsup_{i\to\infty}d(f_{\w_i}(x),f_{s_i}(y))>0);\]
\emph{scrambled of type 2}
if there exists an infinite sequence $\{\w_i\}$ with $\w_i\in\calL^i(\La)$
\[\liminf_{i\to\infty}d(f_{\w_i}(x),f_{\w_i}(y))=0~ ,~\limsup_{i\to\infty}d(f_{\w_i}(x),f_{\w_i}(y))>0.\]
A subset $C$ of $X$ is called \emph{scrambled set of type $1,2$} respectively if any two distinct points $x, y\in C$ respectively form a scrambled pair of type $1,2$. The switched system $(X,\calG|_\La)$ is called \emph{Li-Yorke chaotic of type $1,2$} respectively if there exists an uncountable scrambled set of type $1,2$ respectively.

We can deduce the following result directly from the definition of Xiong chaotic set.
\begin{prop}
Let $(X,\calG|_\La)$ be a switched system, $Q$ a closed subset of $X$. Assume that $Q$ is nondegenerate.
If $C$ is a Xiong chaotic set of type $1,2$ respectively with respect to $Q$, then $C$ is a scrambled set of type $1,2$ respectively.
\end{prop}

\begin{thmx}\label{thm:cha-of-mixing}
Let $(X,\calG|_\La)$ be a switched system and $Q$ a nondegenerate subset of $X$. If $K$ is a perfect compact subset of $X$ then $K$ respectively is weakly mixing of type $1,2$ with respect to $Q$ if and only if there exists $F_\sigma$ set which is Xiong chaotic of type $1,2$ with respect to $Q$ and dense in $K$ respectively.
\end{thmx}

Since the proofs for the two cases are similar, we only give the proof of type $2$.

\begin{proof}
The necessity part comes from lemmas~\ref{lem:COKM};~\ref{lem:SRS};
and ~\ref{lem:SOSC}, and the sufficiency part comes from the following lemma~\ref{lem:sufficiency-of-mixing}.
\end{proof}

\begin{lem}\label{lem:sufficiency-of-mixing}
Under the conditions of Theorem~\ref{thm:cha-of-mixing}. If $K$ is Xiong chaotic of type $1,2$  with respect to $Q$ respectively, then it is weakly mixing of type $1,2$ with respect to $Q$ respectively.
\end{lem}

\begin{proof}
Fix $n\in\N$. Suppose $U_1,U_2,\ldots,U_n,V_1,V_2,\ldots,V_n$ are $2n$ non-empty subsets of $X$ and $U_i\cap K\neq\emptyset$ and $V_i\cap Q\neq\emptyset$, $i=1,\ldots,n$. Choose $x_i\in U_i\cap K$, $y_i\in V_i\cap Q$, and define a map $g: \{x_1,...,x_n\} \rightarrow Q$ by $g(x_i)=y_i$. Then, by the definition of Xiong chaotic set of type $2$ with respect to a given set, there exist $k,\varepsilon$ and $\w\in\calL^k(\La)$ such that $f_{\w_k}((x_i)\in B(y_i,\varepsilon)\subset V_i$. So $\cap_{i=1}^{n}N(K\cap U_i,V_i)\neq\emptyset$. This completes the proof.
\end{proof}

\section{Necessity of Theorem~\ref{thm:cha-of-mixing}}
We will use the methods developed in~\cite{Oprocha2013Coherent} to obtain the necessity of Theorem~\ref{thm:cha-of-mixing}. First, we introduce some tools we need in hyperspace.

Let $(X,d)$ be a compact metric space and $A$ a non-empty subset of $X$.
We say $A$ is \emph{totally disconnected} if its only connected subsets are singletons; \emph{perfect} if it is closed and has no
isolated point; \emph{Cantor's} if it is a compact, perfect, and totally disconnected set; \emph{residual} if it contains a dense $G_\delta$ set. We write $B(x,\varepsilon)=\{y\in X,~d(x,y)<\varepsilon\}$, $d(x,A)=\inf\{d(x,a),a\in A\}$, $B(A,\varepsilon)=\{x\in X:d(x,A)<\varepsilon\}$. Denote by $\overline{B}(A,\varepsilon)$ the closure of $B(A,\varepsilon)$.

Recall that the \emph{hyperspace} $2^X$ of $X$ is the collection of all non-empty closed subsets of $X$ endowed with the \emph{Hausdorff metric} $d_H$ defined by
\[d_H(A,B)=inf\{\varepsilon>0:\overline{B}(A,\varepsilon)\supseteq B~ {and}~\overline{B}(B,\varepsilon)\supseteq A\}.\]
The following family
\[\{\langle U_1,\ldots,U_n\rangle: U_1,\ldots,U_n ~are~ non\text{-}empty~ open~ subsets~ of~ X,~ n\in \N\}\]
forms a basis for a topology of $2^X$ called the \emph{Vietoris topology}, where
\[\langle U_1,\ldots,U_n\rangle:=\{A\in 2^X:A\subset\cup_{i=1}^n U_i,~ {and}~ U_i\cap A\neq\emptyset, ~ {for}~ i=1,\ldots,n\}\]
It is well known that the Hausdorff metric $d_H$ is compatible with the Vietoris
topology for $2^X$ (for detalis see \cite{Nadler1978Hyperspaces}).
A subset $Q$ of $2^X$ is called \emph{hereditary} if $2^A \subset Q$ for every set $A\in Q$.

The following result is a key to the proof for the necessity, which is a consequence of Kuratowski-Mycielski Theorem (see Theorem 5.10 of ~\cite{Akin2004Lectures}).
\begin{lem}\label{lem:COKM}
  Suppose that $X$ is a perfect compact space. If a hereditary subset $Q$ of $2^X$ is residual then there exists a countable Cantor sets $C_1\subset C_2 \subset\cdots$ of $X$ such that $C_i \in Q $ for every $i \geq 1$ and $C =\cup_{i=1}^\infty C_i $ is dense in $X$.
\end{lem}

Let $(X,\calG|_\La)$ be a switched system and $Q$ be a subset of $X$. Given  $\varepsilon>0$, we say that a subset $A$ of $X$ is $\varepsilon$-spread in $Q$ if there exists $\delta\in(0,\varepsilon)$, $z_1,z_2,\ldots,z_n\in X$ such that $A\subset\cup_{i=1}^nB(z_i,\delta)$ and for any map $h:~\{z_1,z_2,\ldots,z_n\}\rightarrow Q$, there exists $k\in\N$ with $\frac{1}{k}<\varepsilon$ and $\w\in\calL^k(\La)$ such that $f_{\w}(B(z_i,\delta)\subset B(h(z_i),\varepsilon)$.

Denote by $\mathfrak{X}(\varepsilon,Q)$ the collection of all compact sets $\varepsilon$-spread in $Q$. Then $\mathfrak{X}(\varepsilon,Q)$ is hereditary. In fact, if $A$ is $\varepsilon$-spread in $Q$ and $B$ is a non-empty closed subset of $A$, then $B$ is also $\varepsilon$-spread in $Q$. Set
\[\mathfrak{X}(Q)=\bigcap_{p=1}^\infty\mathfrak{X}(\frac{1}{p},Q).\]

\begin{lem}\label{lem:SRS}
 Suppose $(X,\calG|_\La)$ is a switched system, $Q$ is a closed subset of $X$, and $K\subset X$. If $K$ is compact and weakly mixing of type $2$ with respect to $Q$, then $\mathfrak{X}(Q)\cap2^K$ is a residual subset of $2^K$.
\end{lem}

\begin{proof}
  First, we show that $\mathfrak{X}(\frac{1}{p},Q)\cap 2^K$ is open in $2^K$. Let $A\in\mathfrak{X}(\frac{1}{p},Q)\cap 2^K$. Then there exists $\delta>0$ and $z_1,z_2,\cdots,z_n\in X$ satisfying the definition of $\frac{1}{p}$-spread in $Q$. Let $V_i=B(z_i,\delta)$, for $i=1,2,\cdots,n$. It is easy to verify that every $B\subset\cup_{i=1}^nV_i$ is also $\frac{1}{p}$-spread in $Q$. Specially, if $B\in\langle V_1\cap K,\ldots,V_n\cap K\rangle$, then $B\in \mathfrak {X}(\frac{1}{p},Q)$. So $\mathfrak {X}(\frac{1}{p},Q)\cap 2^K$ is open in $2^K$.

  Next, to prove $\mathfrak {X}(\frac{1}{p},Q)\cap2^K$ is dense in $2^K$, we shall show that for any non-empty open sets $U_1,U_2,\ldots,U_n$ of $X$ intersecting $K$, $\langle U_1\cap K,U_2\cap K,\ldots,U_n\cap K\rangle\cap\mathfrak {X}(\frac{1}{p},Q)\cap2^K\neq\emptyset$.
Since $Q$ is compact, there is a finite subset $\{y_1,y_2,\ldots,y_m\}$ of $Q$ such that $Q\subset\cup_{i=1}^mB(y_i,\frac{1}{2p})$. For convenience, denote $V_i=B(y_i,\frac{1}{2p})$ for $i=1,2,\ldots,m$ and $\mathcal{S}=\{\alpha_j=(a_1,a_2,\ldots,a_n):a_i\in\{1,2,\ldots,m\},j=1,2,\ldots,m^n\}$. Let $\alpha_j(i)$ denote the $i$-th component of $\alpha_j$. As $K$ is a weakly mixing set with respect to $Q$, there exists $k_1$ and $\w_1\in\calL^{k_1}(\La)$ such that
  \[f_{\w_1}(U_i\cap K)\cap V_{\alpha_1(i)}\neq\emptyset,~i=1,2,\ldots n.\]
  Choose a non-empty open subset $W_i^1$ of $U_i$ intersecting $K$ such that $f_{\w_1}(W_i^1)\subset V_{\alpha_1(i)}$. For $W_i^1$, there exists $k_2$ and $\w_2\in\calL^{k_2}(\La)$ such that for $i=1,2,\ldots n$
  \[f_{\w_2}(W_i^1\cap K)\cap V_{\alpha_2(i)}\neq\emptyset.\]
  Choose a non-empty open subset $W_i^2$ of $W_i^1$ intersecting $K$ such that $f_{\w_2}((W_i^2)\subset V_{\alpha_2(i)}$. Using this process all over the sequences of $\mathcal {S}$, we get $k_1,k_2,\ldots,k_{m^n}$ and $\w^j$, $1\leq j\leq m^n$ such that
  \[W_{i}^{m^n}\subset W_{i}^{m^n-1}\subset\cdots\subset W_{i}^1\subset U_i\]
  and
  \[f_{\w_j}(W_{i}^{k_{m^n}})\subset V_{\alpha_j(i)},j=1,2,\ldots,m^n.\]
  Pick $z_i\in W_{i}^{k_{m^n}}\cap E$. Then $\{z_1,z_2,\ldots,z_n\}\in\langle U_1\cap E,U_2\cap E,\ldots,U_n\cap E\rangle\cap 2^E$. Choose $\delta\in(0,\frac{1}{p})$ such that $B(z_i,\delta)\subset W_i^{k_{m^n}}$ for $i=1,2,\ldots,n$. For any map $h:\{z_1,z_2,\ldots,z_n\}\rightarrow Q$, there exists $\alpha_j$ such that $V_{\alpha_j(i)}\subset B(h(z_i),\frac{1}{p})$ for $i=1,2,\ldots,n$. Therefore,
  \[f_{\w_j}(B(z_i,\delta)\subset f_{\w_j}(W_i^{k_{m^n}})\subset V_{\alpha_j(i)}\subset B(h(z_i),\frac{1}{p}) \]
  for $i=1,2,\ldots,n$, which implies $\{z_1,z_2,\ldots,z_n\}$ is $\frac{1}{p}$-spread in $Q$ and $\mathfrak{X}(\frac{1}{p},Q)\cap2^K$ is dense in $2^K$. It follows that $\mathfrak{X}(Q)\cap 2^K$ is a residual subset of $2^K$.
\end{proof}




\begin{lem}\label{lem:SOSC}
  Let $(X,\calG|_\La)$ be a switched system, and let $Q$ be a closed subset of $X$. If $C_1\subset C_2\subset\cdots$ is an increasing unbounded sequence of sets in $\mathfrak{X}(Q)$, then for any subset $A$ of $C:=\cup_{i=1}^\infty C_i$ and continuous function $h:A\rightarrow Q$ there exists a sequence $\{k_i\}_{i=1}^\infty$ of positive integers and $\w_i\in\calL^{k_i}(\La)$ such that
  \[\lim_{i\rightarrow\infty}f_{\w_i}(x)=h(x)\]
  for every $x\in A$.
\end{lem}

\begin{proof}
  For $i\in\N$, $C_i$ is $\frac{1}{i}$-spread in $Q$, so there exist $\delta_i<\frac{1}{i}$ and $z_1^i,z_2^i,\cdots,z_{n_i}^i$ such that for any map $h:C_i\rightarrow Y$ there exists $k_i$ and $\w_i\in\calL^{k_i}(\La)$ satisfying $C_i\subset\cup_{m=1}^{n_i}B(z_m^i,\delta_i)$ and $f_{\w_i}(B(z_m^i))\subset B(h(z_m^i),\frac{1}{i})$ for $m=1,2,\cdots,n_i$. We shall show that $\{k_i\}_{i=1}^\infty$ and $\{\w_i\}_{i=1}^\infty$ are required.

  For any $x\in A$, there exists $l_x$ such that $x\in C_i$ for all $i>l_x$. By  continuity of $h$, for every $\varepsilon>0$ there is $l_\varepsilon>\frac{2}{\varepsilon}$ such that if $d(x,y)<\frac{1}{l_\varepsilon}<\frac{\varepsilon}{2}$, then
  \[d(h(x),h(y))<\frac{\varepsilon}{2}.\]
  For every $i>\max\{l_x,l_\varepsilon\}$, there exists $\delta_i<\frac{1}{i}<\frac{1}{l_\varepsilon}<\frac{\varepsilon}{2}$, $z_1^i,z_2^i,\cdots,z_{n_i}^i$ such that
  \[f_{\w_i}((B(z_m^i,\delta_i))\subset B(h(z_m^i),\frac{1}{i}),\forall m=1,2,\cdots,n_i,\]
  and there exists $z_m^i$ such that $x\in B(z_m^i,\delta_i)$. Then
  \[d(f_{\w_i}(x),h(x))<d(f_{\w_i}(x),h(z_m^i))+d(h(z_m^i),h(x))<\varepsilon,\]
  which means that
  \[\lim_{i\rightarrow\infty}f_{\w_i}(x)=h(x).\]
\end{proof}

\begin{exa}
Let $X=\R$, $f_0=2x$, $f_1=2-2x$, and $Q=[0,1]$. Then $Q$ is weakly mixing of type $1$ with respect to itself for switched system $(X,\calG)$, where $\calG=\{f_0,f_1\}$.
\end{exa}

\begin{proof}
It is clear that the the following tent map:
\[f(x)=\left\{
         \begin{array}{ll}
           2x, & x\in[0,\frac{1}{2}]; \\
           2-2x, & x\in(\frac{1}{2},1].
         \end{array}
       \right.
\]
is weakly mixing (see Example 3.1.3 in~\cite{Kolyada1997Some}). So for any nonempty open subsets $U_1,U_2,\ldots,U_n$ and $V_1,V_2,\ldots,V_n$ of $Q$, there exists $m\in\N$ such that
\[f^m(U_i)\cap V_i\neq\emptyset.\]
We shall construct $n$ elements $\w_i\in\{0,1\}^m$ such that
\[f_{\w_i}(U_i)\cap V_i\neq\emptyset.\]
Let $x_i\in U_i\cap f^{-n}(V_i)$. If $x_i\in U_i\cap[0,\frac{1}{2}]$ then we replace $f$ by $f_0$; otherwise, replace $f$ by $f_1$. So for every $x_i$ there exists $w_i\in\{0,1\}$ such that $f_{w_i}(x)\in Q$. Continuing the process, we can get the desired result.
\end{proof}


\textbf{Acknowledgements}  This work was supported by National Natural Science Foundation of China (Grant No. 11771459, 11471125) and International Program for Ph.D. Candidates, Sun Yat-Sen University. The second author would like to thank Prof. X. Zou and his department for their hospitality during his visit to Western University. A part of this work was done when the second author was visiting Western University.

\end{document}